\documentclass[a4,12pt]{article}%

\usepackage{graphicx}
\usepackage{graphics}

\usepackage{amsmath}
\usepackage{amsfonts}
\usepackage{amssymb}
\usepackage{enumerate}
\usepackage{xypic}
\newtheorem{theorem}{Theorem}[section]

\newtheorem{definition}[theorem]{Definition}

\newenvironment{proof}[1][Proof]{\textbf{#1.} }{\ \rule{0.5em}{0.5em}}

\newcommand{\E}{{\rm \bf E}}

\newcommand{\prob}{{\rm \bf P}}

\newcommand{\dN}{{\bf N}}

\newcommand{\dR}{{\bf R}}

\newcommand{\calF}{{\cal F}}

\newcommand{\ep}{\varepsilon}

\newcounter{figurecounter}
\begin{document}

\title{Jointly Controlled Lotteries with Biased Coins%
\thanks{The authors thank Johannes H\"orner, Ehud Lehrer, and Nicolas Vieille for useful discussion.
E. Solan acknowledges the support of the Israel Science Foundation, grant \#217/17.}}

\author{Eilon Solan,%
\thanks{The School of Mathematical Sciences, Tel Aviv
University, Tel Aviv 6997800, Israel. e-mail: eilons@post.tau.ac.il.}
\ Omri N.~Solan,%
\thanks{The School of Mathematical Sciences, Tel Aviv
University, Tel Aviv 6997800, Israel. e-mail: omrisola@post.tau.ac.il.}
\ and Ron Solan%
\thanks{e-mail: ron\_solan@walla.com.}}

\maketitle

\begin{abstract}
We provide a mechanism that uses two biased coins
and implements any distribution on a finite set of elements, in such a way that even if the outcomes of one of the coins is determined by an adversary,
the final distribution remains unchanged.
We apply this result to show that every quitting game in which at least two players have at least two continue actions
has an undiscounted $\ep$-equilibrium, for every $\ep > 0$.
\end{abstract}

\noindent
\textbf{Keywords:} Jointly controlled lotteries, biased coin, quitting games, equilibrium.

\noindent
\textbf{JEL Classification Numbers:} D82, C72.

\section{Introduction}

Random numbers are fundamentals for almost all secure computer systems.
However, random number generators are prone to attacks by adversaries, who may attempt to control their outputs.
To hamper an attacker,
one can use several random number generators, and devise a mechanism that uses the outcomes of all generators to produce one random outcome,
in such a way that if an attacker controls the outcomes of one (or more, but not all) of the random number generators,
the distribution of the mechanism's outcome does not change.
For example, if each random number generator chooses a bit according to the uniform distribution,
and the goal is to produce a random bit whose distribution is uniform,
then a plausible mechanism is to output the xor of all input bits.

In this paper we study this problem, when the distribution according to which each random generator device chooses its output is not the uniform distribution.
More formally, we are given $k$ random generator devices; in every instance device $i$ chooses a letter from a finite alphabet $A_i$ according to the probability distribution $\prob_i$,
independent of past choices.
A mechanism is given by a stopping time $\tau$ and a deterministic rule $f$
that dictates which letter in some finite alphabet $J$ is selected based on the letters that were selected by the random generator devices up to time $\tau$.
Given a probability distribution $\nu$ over $J$,
the goal is to device a mechanism that is immune to attacks by an adversary: even if $k-1$ of the random number generators become faulty
and produce letters according to some law,
the distribution of the outcome of the mechanism is still $\nu$.

We will provide two mechanisms for selecting an element of $J$,
both of which depend on a parameter $\ep > 0$.
One mechanism has a bounded length and selects an element in $J$ with a distribution $\ep$-close to $\nu$:
as long as at least one device is not faulty,
the probability that each element $j \in J$ is selected is $\ep$-close to $\nu(j)$.
The second mechanism may be unbounded yet it is finite a.s.~and selects each element $j \in J$ with probability exactly $\nu(j)$.
Moreover, as long as exactly one device is faulty,
the probability that each element $j \in J$ is selected cannot exceed $\nu(j)$,
and, in case the mechanism never terminates, the identity of the faulty devices is revealed by the information that reached the mainframe.

To demonstrate the usefulness of the result, we apply it to study undiscounted equilibria in stochastic games.
Whether every multiplayer stochastic game admits an undiscounted $\ep$-equilibrium for every $\ep > 0$ is one of the main open problems in game theory to date;
see Flesch, Thuijsman, and Vrieze (1997), Solan (1999), Vieille (2000a,b), Solan and Vieille (2001), Simon (2012), and Solan and Solan (2017) for partial results.
The class of games that we study in this paper is the class of general quitting games.
Those are quitting games in which each player has a single quitting action and may have several continue actions.
This class of games was studied by Solan and Solan (2018),
who showed that those games admit a sunspot $\ep$-equilibrium for every $\ep > 0$;
that is, an $\ep$-equilibrium in an extended game in which at every stage the players observe the outcome of
a uniformly distributed random variable on $[0,1]$, which is independent of past signals and past play.
Using jointly controlled lotteries with biased coins we will show that
if at least two players have at least to continue actions, an undiscounted $\ep$-equilibrium exists.

To date it is not known whether quitting games in which each player has a single quitting action and a single continue action
admit undiscounted $\ep$-equilibria.
Our result shows that when players have enough flexibility in coordinating their play,
an undiscounted $\ep$-equilibrium exists.

There are various ways in which one can strive to extend this equilibrium existence result.
\begin{itemize}
\item
Our method shows that jointly controlled lotteries enable one to transform sunspot $\ep$-equilibria into undiscounted $\ep$-equilibria,
in various settings of stochastic games.
Can one extend the existence result
to other classes of stochastic games that include more than one nonabsorbing state?
\item
One property of the class of general quitting games is that some players have two actions that induce the same transitions,
for every given action profile of the other players.
Is it true that an undiscounted $\ep$-equilibrium exists in any stochastic game in which for every state~$s$,
every player~$i$ and every action~$a_i$ of player~$i$,
there is an action $a'_i \neq a_i$ that yields the same transition as $a_i$ at state~$s$?
\end{itemize}

Biased coins are not prevalent in game theory, since usually it is assumed that players have all randomization means that they need.
One exception is Gossner and Vieille (2002),
who studied two-player zero-sum repeated games in which the randomization device of one of the players is a biased coin that
he can toss once at the beginning of every stage.
They showed that the player can do better than using at every stage the outcome of the toss performed at the beginning of that stage,
and characterized the value of the game as a function of the distribution of the coin.
In their model, the player need not use the information provided by the coin at the stage in which it is obtained,
but may rather use this information in subsequent stages.
In our model, in contrast, aggregating the random information is impossible,
since the letter chosen by a faulty device may depend on past choices of the unfaulty device.

Though jointly controlled lotteries with biased coins reminds one of mediated talk (see, e.g., Lehrer (1996) and Lehrer and Sorin (1997))
and cheap talk (see, e.g., Farrell and Rabin (1996) and Aumann and Hart (2003)),
there are some significant differences among the models.
Indeed, while in mediated talk and cheap talk the players are free to select the messages they send out and the goal is to choose an action for each player,
in our model, when unfaulty, the devices choose messages according to a known stationary probability distribution
and the goal is to choose one outcome.

The paper is arranged as follows.
In Section~\ref{section:lottery} we discuss jointly controlled lotteries with biased coins,
and in Section~\ref{section:game} we apply the mechanism of jointly controlled lotteries to general quitting games.

\section{Jointly Controlled Lotteries with Biased Coins}
\label{section:lottery}

To simplify the presentation we will assume that there are two random number generators;
the extension to any number of random number generators follows the same lines.

Let $A_1$ and $A_2$ be two finite sets, each containing at least two elements.
The set of \emph{finite histories}%
\footnote{By conventions, the set $(A_1 \times A_2)^0$ contains only the empty history.}
is $H := \cup_{t=0}^\infty (A_1 \times A_2)^t$,
and the set of \emph{infinite histories} is $H^\infty := (A_1 \times A_2)^\infty$.
The set $H^\infty$ is a measurable space when equipped with the product $\sigma$-algebra.
For every $t \geq 0$ denote by $\calF^t$ the $\sigma$-algebra over $H^\infty$ defined by all histories of length $t$;
it is the $\sigma$-algebra spanned by the sets $C(h^t) := \{ h = (a^1,a^2,\cdots) \in H^\infty \colon h^t = (a^1,\cdots,a^t)\}$
for $h^t \in (A_1 \times A_2)^t$.

The basic concept that we need is that of a mechanism,
which describes how to generate an element from a set $J$ given an infinite history.
\begin{definition}
A \emph{mechanism} is a triplet $M := (\tau,J,f)$ where
\begin{itemize}
\item
$\tau$ is a stopping time w.r.t.~the filtration $(\calF^t)_{t \geq 0}$.
\item
$J$ is a finite set.
\item
$f : H^\infty \to J$ is a function that is measurable w.r.t.~the $\sigma$-algebra $\calF^\tau$.
\end{itemize}
\end{definition}
%Since the function $f$ is $\calF^\tau$-measurable, we can view it as a function that is defined over the set of histories
%$\{ (a^1,a^2,\cdots,a^{\tau(h)}) \colon h = (a^1,a^2,\cdots) \in H^\infty\}$.

When $M = (\tau,J,f)$ is a mechanism, every probability distribution $\mu$ over $H^\infty$ defines a probability distribution $\mu_M$ over $J$ by
\[ \mu_M(j) = \mu(\{ h \in H^\infty \colon f(h) = j\}). \]

Let $i \in \{1,2\}$.
A \emph{(behavior) strategy} for the $i$'s coordinate is a function $\sigma_i : H \to \Delta(A_i)$
that assigns a distribution over $A_i$ to each finite history.
The set of all strategies for the $i$'th coordinate is denoted $\Sigma_i$.
A strategy is \emph{stationary} if $\sigma_i(h^t)$ is independent of $h^t \in H$.
Every pair of strategies $(\sigma_1,\sigma_2)$ defines a probability measure $\prob_{\sigma_1,\sigma_2}$ over $H^\infty$,
and in particular,
together with a mechanism $M$ it defines a probability measure $\prob_{\sigma_1,\sigma_2,M}$ over $J$.

We now present three properties of mechanisms: having finite length,
being able to implement a given probability distribution,
and being able to implement the distribution in a secure way.

\begin{definition}
Let $T \in \dN$.
A mechanism $M = (\tau,J,f)$ \emph{has length at most $T$} if $\prob_{\sigma_1,\sigma_2}(\tau \leq T) = 1$
for every pair of strategies $(\sigma_1,\sigma_2)$.
\end{definition}

\begin{definition}
Let $\ep \geq 0$ and let $\nu$ be a probability distribution over $J$.
The mechanism $M = (\tau,J,f)$ and the pair of strategies $(\sigma_1,\sigma_2) \in \Sigma_1\times\Sigma_2$ \emph{$\ep$-implement the distribution $\nu$} if
$\| \prob_{\sigma_1,\sigma_2,M} - \nu\|_\infty \leq \ep$.
The mechanism $M = (\tau,J,f)$ and the pair of strategies $(\sigma_1,\sigma_2) \in \Sigma_1\times\Sigma_2$
\emph{$\ep$-implement the distribution $\nu$ in a strong secure fashion}
if for every $i \in \{1,2\}$ and every strategy $\sigma'_i \in \Sigma_i$,
the mechanism $M$ and the pair of strategies $(\sigma'_i,\sigma_{3-i})$ $\ep$-implement the distribution $\nu$.
\end{definition}

Our first result concerns the possibility of $\ep$-implementing any distribution in a secure fashion given any pair of stationary strategies.
\begin{theorem}
\label{theorem:strong}
Let $A_1$, $A_2$, and $J$ be three finite sets, each of which contains at least two elements.
Let $\sigma_1$ (resp.~$\sigma_2$) be a stationary strategy that selects all elements in $A_1$ (resp.~$A_2$) with positive probability,
and let $\nu$ be any distribution on $\nu$.
For every $\ep > 0$ there is a mechanism $M = (\tau,J,f)$ that has a finite length and, together with the pair of stationary strategies $(\sigma_1,\sigma_2)$,
$\ep$-implements $\nu$ in a strong secure fashion.
\end{theorem}

\begin{proof}
Assume w.l.o.g.%
\footnote{If the set $A_i$ contains more than two elements, divide it arbitrarily into two subsets,
and treat all elements that lie in the same subset as equivalent.}
that $|A_1| = |A_2| = 2$,
and denote $A_i = \{\alpha,\beta\}$ for $i \in \{1,2\}$.
For each $i\in\{1,2\}$, the strategy $\sigma_i$ is stationary;
denote by $\sigma_i(\alpha)$ and $\sigma_i(\beta)$ the per-stage probability that strategy $\sigma_i$ selects the elements $\alpha$ and $\beta$, respectively.

For every $t \in \dN$ define a random variable $Y^t$ over $H^\infty$ as follows:
\[ Y^t := \left\{
\begin{array}{lll}
-\sigma_1(\beta) \sigma_2(\beta) & \ \ \ \ \ & a^t = (\alpha,\alpha),\\
\sigma_1(\beta) \sigma_2(\alpha) & \ \ \ \ \ & a^t = (\alpha,\beta),\\
\sigma_1(\alpha) \sigma_2(\beta) & \ \ \ \ \ & a^t = (\beta,\alpha),\\
-\sigma_1(\alpha) \sigma_2(\alpha) & \ \ \ \ \ & a^t = (\beta,\beta),\\
\end{array}
\right. \]
We observe that $\E_{\sigma_1,\sigma'_2}[Y^t] = \E_{\sigma'_1,\sigma_2}[Y^t] = 0$,
for every $t \geq 0$ and every pair of strategies $(\sigma'_1,\sigma'_2) \in \Sigma_1\times\Sigma_2$.

For every real number $C > 0$ let $\tau_C$ be the stopping time
\begin{equation}
\label{def:tau}
\tau_C := \min\left\{ t \in \dN \colon \sum_{k=1}^t (Y^k)^2 \geq C\right\}.
\end{equation}
Denoting by $c_0 := \min\{ \sigma_1(\alpha), \sigma_1(\beta), \sigma_2(\alpha), \sigma_2(\beta)\} > 0$,
we obtain that the stopping time $\tau_C$ is bounded by $\tfrac{C}{c_0^2}$.
Denote $Z_C := \frac{\sum_{t=1}^{\tau_C}Y^t}{\sqrt{C}}$.
The Martingale Central Limit Theorem (see, e.g., McLeish, 1974),
implies that for each player~$i$ and each strategy $\sigma'_i\in \Sigma_i$ of player~$i$,
under the pair of strategies $(\sigma'_i,\sigma_{3-i})$ the distribution of $Z_C$ converges to the standard normal distribution as $C$ goes to infinity.
Moreover, the rate of convergence is independent of $\sigma'_i$.

It follows that to $\ep$-implement $\nu$ in a strong secure fashion,
we need to divide the real line $\dR$ into $J$ disjoint intervals $I_1,I_2,\cdots,I_J$,
such that the probability of the interval $I_j$ under the standard normal distribution is $\nu(j)$, for each $j \in J$.
We then choose $C$ sufficiently large, and define the mechanism $M$ by $(\tau_C,J,f)$,
where for every infinite history $h$ we define $f(h)$ to be the unique $j \in J$ such that $Z_C(h) \in I_j$.
\end{proof}

\bigskip

We now weaken the security requirement of the mechanism.
The weaker condition does not require that the mechanism stops in finite time whatever the players play,
but rather that it stops in finite time when the two random generator devices are not faulty,
and that if one of the devices is faulty,
then its outputs will necessarily reveal that it is faulty.

\begin{definition}
\label{def:weak}
Let $\ep \geq 0$ and let $\nu$ be a probability distribution over $J$.
The mechanism $M = (\tau,J,f)$ and the pair of strategies $(\sigma_1,\sigma_2) \in \Sigma_1\times\Sigma_2$
\emph{$\ep$-implement the distribution $\nu$ in a weak secure fashion}
if the following conditions hold:
\begin{itemize}
\item[(W.1)]   $M$ and $(\sigma_1,\sigma_2)$ $\ep$-implement the distribution $\nu$.
\item[(W.2)]   For every strategy $\sigma'_1 \in \Sigma_1$ we have $\prob_{\sigma'_1,\sigma_2,M}(j) \leq \nu(j)$.
\item[(W.3)]   For every strategy $\sigma'_2 \in \Sigma_2$ we have $\prob_{\sigma_1,\sigma'_2,M}(j) \leq \nu(j)$.
\item[(W.4)]   There are two disjoint events $D_1$ and $D_2$ such that
\begin{itemize}
\item
$D_1 \cup D_2 \subseteq \{\tau = \infty\}$.
\item
$\prob_{\sigma_1,\sigma_2}(D_1) = \prob_{\sigma_1,\sigma_2}(D_2) = 0$.
\item
For every strategy $\sigma'_1 \in \Sigma_1$ we have $\prob_{\sigma'_1,\sigma_2}(D_1) + \prob_{\sigma'_1,\sigma_2}(\{\tau < \infty\}) = 1$.
\item
For every strategy $\sigma'_2 \in \Sigma_2$ we have $\prob_{\sigma_1,\sigma'_2}(D_2) + \prob_{\sigma_1,\sigma'_2}(\{\tau < \infty\}) = 1$.
\end{itemize}
\end{itemize}
\end{definition}
The event $D_1$ and $D_2$ in Definition~\ref{def:weak} are used to reveal the identity of the faulty device:
on the event $D_i$ it is known that device~$i$ is faulty, for $i=1,2$;
indeed, this set occurs with probability 0 if no device is faulty,
and it occurs whenever the mechanism does not stop and device~$i$ is faulty.
Note that whereas strong security requires the stopping time $\tau$ to be uniformly bounded,
weak security has no such restriction.

\begin{theorem}
\label{theorem:weak}
Let $A_1$, $A_2$, and $J$ be three finite sets, each of which contains at least two elements.
Let $\sigma_1$ (resp.~$\sigma_2$) be a stationary strategy that selects all elements in $A_1$ (resp.~$A_2$) with positive probability,
and let $\nu$ be any distribution on $\nu$.
There is a mechanism $M = (\tau,J,f)$ that $0$-implements $\nu$ in a weak secure fashion.
\end{theorem}

\begin{proof}
Assume w.l.o.g.~that $|A_1| = |A_2| = 2$,
and denote $A_i = \{\alpha,\beta\}$ for $i=1,2$.
Let $(Y^t)_{t \in \dN}$ be a stochastic process with values in $\Delta(J)$, adapted to the filtration $(\calF^t)_{t \geq 0}$, which satisfies the following properties:
\begin{itemize}
\item[(C.1)]   $Y^0 = \nu$.
\item[(C.2)]   $Y^{t+1}$ depends deterministically on $Y^t$, $a_1^{t+1}$, and $a_2^{t+1}$, and not on $Y^0,\cdots,Y^{t-1}$.
\item[(C.3)]   $\E_{\sigma_1,\alpha}[Y^{t+1} \mid Y^t] = \E_{\sigma_1,\beta}[Y^{t+1} \mid Y^t] = Y^t$.
\item[(C.4)]   $\E_{\alpha,\sigma_2}[Y^{t+1} \mid Y^t] = \E_{\beta,\sigma_2}[Y^{t+1} \mid Y^t] = Y^t$.
\item[(C.5)]   If the support of $Y^t$ contains more than one element,
then for every possible value $\lambda$ of the random variable $Y^t$ that is attained with positive probability,
at least one of the distributions $(Y^{t+1} \mid Y^t,\alpha,\alpha)$,
$(Y^{t+1} \mid Y^t,\alpha,\beta)$, $(Y^{t+1} \mid Y^t,\beta,\alpha)$, and $(Y^{t+1} \mid Y^t,\beta,\beta)$ has a support that contains less elements than the support of $\lambda$.
\end{itemize}
To show that such a process exists, let $\lambda$ be a possible value of $Y^t$.
Denote by $d_{a_1,a_2}$ the distribution $(Y^{t+1} \mid Y^t=\lambda,a_1,a_2)$, for each $a_1,a_2 \in \{\alpha,\beta\}$.
Conditions~(C.3) and~(C.4) determine three equalities that the four variables $(d_{a_1,a_2})_{a_1,a_2 \in \{\alpha,\beta\}}$ should satisfy.
One solution of these equalities is $d_{a_1,a_2} = \lambda$ for every $a_1,a_2 \in \{\alpha,\beta\}$.
Since the number of conditions is smaller by one than the number of variables,
the set of solutions is a line,
hence there is a solution on the boundary of the set $(\Delta(J))^4$,
and therefore indeed such a stochastic process $(Y^t)_{t \geq 0}$ exists.

Conditions~(C.3) and~(C.4) imply that the process $(Y^t)_{t \in \dN}$ is a martingale under $(\sigma_1,\sigma_2)$,
hence it converges $\prob_{\sigma_1,\sigma_2}$-a.s. to a random variable $Y^\infty$.
Denote
\[ c_1 := \min\{\sigma_1(\alpha)\sigma_2(\alpha), \sigma_1(\alpha)\sigma_2(\beta), \sigma_1(\beta)\sigma_2(\alpha), \sigma_1(\beta)\sigma_2(\beta)\} > 0. \]
Under the stationary strategy pair $(\sigma_1,\sigma_2)$, for every $t \in \dN$,
the probability that the support of $Y^{t+1}$ is strictly contained in the support of $Y^t$ is at least $c_1$.
It follows that $Y^\infty$ is a Dirac measure $\prob_{\sigma_1,\sigma_2}$-a.s.
Since the process $(Y^t)_{t \in \dN}$ is a martingale, it follows that for every $j\in J$ we have $\prob_{\sigma_1,\sigma_2}(Y^\infty = j) = Y^0(j) = \nu(j)$.
Setting $\tau = \infty$ and $M = (\tau,J,Y^\infty)$ we obtain that $M$ 0-implements the distribution $\nu$,
and Condition~(W.1) holds.

Condition~(C.3) implies that the process $(Y^t)_{t \geq 0}$ is a martingale under $(\sigma_1,\sigma'_2)$ for every strategy $\sigma'_2 \in \Sigma_2$,
which implies that Condition~(W.3) holds.
Analogously, Condition~(W.2) holds as well.

We complete the proof by proving that Condition~(W.4) holds.
Denote by $(\widehat a^t_1,\widehat a^t_2) \in A_1 \times A_2$ an action pair such that the support of $(Y^{t+1} \mid Y^t,\widehat a^t_1,\widehat a^t_2)$
is strictly contained in the support of $Y^t$.
We note that under strategy $\sigma_i$ we have
\[ \prob_{\sigma_i,\sigma'_{3-i}}(a^t_i = \widehat a^t_i \hbox{ infinitely often}) = 1, \ \ \ \forall \sigma'_{3-i}\in \Sigma_{3-i}. \]
For $i\in\{1,2\}$ define an event $D_i$ by
\[ D_i := \{ a^t_i = \widehat a^t_i \hbox{ finitely many times}, a^t_{3-i} = \widehat a^t_{3-i} \hbox{ infinitely often}\}. \]
The event $D_i$ contains all histories in which device~$3-i$ chooses the action that leads to a decrease in the support of $Y^t$ infinitely many times,
while device~$i$ does not do so.
The reader can verify that Condition~(W.4) in Definition~\ref{def:weak} holds,
and therefore the mechanism $M$ 0-implements the distribution $\nu$ in a weak secure fashion.
\end{proof}

%\bigskip
%
%We note that the construction in the proof of Theorem~\ref{theorem:weak} can be adapted to prove Theorem~\ref{theorem:strong}.
%Indeed, let $\ep > 0$ be sufficiently small and let $\tau$ be the stopping time defined by
%\[ \tau := \min\{ t \geq 0 \colon Y^t(j) \geq 1-\ep \ \ \ \hbox{ for some } j \in J\}. \]
%One can show that $\|Y^{t+1}-Y^t\|_\infty$ is uniformly bounded from 0,
%hence $\tau$ is uniformly bounded.
%One can then prove that the mechanism $(\tau,J,f)$ $g(\ep)$-implements the distribution $\nu$ in the strong sense,
%where $f(h)$ is the unique element $j$ in $J$ that satisfies $Y^{\tau}(j) \geq 1-\ep$.

\section{Undiscounted $\ep$-Equilibrium in General Quitting Games}
\label{section:game}

In this section we provide an application of jointly controlled lotteries with biased coins to the area of stochastic games.
As mentioned in the introduction, whether every stochastic game admits an undiscounted equilibrium payoff is one of the most challenging
open problems in game theory to date.
We will use the tools developed in Section~\ref{section:lottery} to prove the existence of
an undiscounted $\ep$-equilibrium in a class of stochastic games that was termed
\emph{general quitting game} in Solan and Solan (2018).
A \emph{general quitting game} is a vector $\Gamma = (I,(A_i^c)_{i \in I},u)$ where
\begin{itemize}
\item   $I$ is a finite set of players.
\item   $A_i^c$ is a finite nonempty set of \emph{continue actions},
for each player $i \in I$.
The set of all actions of player~$i$ is $A_i := A_i^c \cup \{Q_i\}$, where $Q_i$ is interpreted as a \emph{quitting action}.
The set of all action profiles is $A = \times_{i \in I} A_i$.
\item   $u : A \to [0,1]^I$ is a payoff function.
\end{itemize}
The game proceeds as follows.
At every stage $t \in \dN$, each player $i \in I$ chooses an action $a_i^t \in A_i$.
Let $a^t = (a_i^t)_{i \in I}$ be the action profile chosen at stage $t$.
Denote by $t_*$ the first stage in which some player selects his quitting action;
that is, the first stage $t$ such that $a^t_i = Q_i$ for some player $i \in I$.
The stage payoff at stage $t$ is given by $u(a^{\min\{t,t_*\}})$.

A \emph{(behavior) strategy} of player~$i$ is a function $\sigma_i \colon \left(\cup_{t=0}^\infty A^t\right) \to \Delta(A_i)$.
A \emph{strategy profile} is a vector of strategies $\sigma = (\sigma_i)_{i \in I}$, one for each player.
Every strategy profile $\sigma$ induces a probability distribution over the set of plays $A^\infty$.
Denote by $\E_\sigma$ the corresponding expectation operator
and by
\[ \gamma(\sigma) := \E_\sigma\left[\lim_{T \to \infty} \frac{1}{T}\sum_{t=1}^T u(a^{\min\{t,t_*\}})\right] \]
the expected (undiscounted) \emph{payoff} under strategy profile $\sigma$.
Note that the way a strategy is defined after the termination stage $t_*$ does not affect the payoff.

A mixed action profile $x \in \times_{i \in I} \Delta(A_i)$ is \emph{nonabsorbing} if under $x$ all players play continue actions
with probability 1, and it is \emph{absorbing} otherwise.

Let $\ep \geq 0$.
A strategy profile $\sigma = (\sigma_i)_{i \in I}$ is an \emph{$\ep$-equilibrium}%
\footnote{The concept that we define is that of undiscounted $\ep$-equilibrium.
Theorem~\ref{theorem:game} below holds also for the stronger notion of uniform $\ep$-equilibrium as well.}
if for every player $i \in I$ and every strategy $\sigma'_i$ of player~$i$,
\[ \gamma_i(\sigma) \geq \gamma_i(\sigma'_i,\sigma_{-i}) - \ep. \]
A \emph{sunspot $\ep$-equilibrium} is an $\ep$-equilibrium in an extended game $\Gamma^E$ that contains a correlation device,
which sends a public signal $s^t$ at the beginning of each stage $t \in \dN$.
Here, $s^t$ is uniformly distributed in $[0,1]$ and independent of $s^1,\cdots,s^{t-1}$ and of the past actions played by the players.
In particular, a \emph{(behavior) strategy} for player~$i$ in the extended game $\Gamma^E$ is a function
$\xi_i \colon \left(\cup_{t=0}^\infty \left(\times_{i \in I} ([0,1] \times A_i)\right)^t\right) \times [0,1] \to \Delta(A_i)$.
The payoff induced by a strategy profile $\xi = (\xi_i)_{i \in I}$ is
\[ \gamma^E(\xi) := \E_\xi\left[\lim_{T \to \infty} \frac{1}{T}\sum_{t=1}^T u(a^{\min\{t,t_*\}})\right], \]
where $\E_\xi$ is the expectation w.r.t.~the probability distribution $\prob_\xi$ induced by $\xi$ over the space of infinite plays $\left(\times_{i \in I} ([0,1] \times A_i)\right)^\infty$.
The strategy profile $\xi$ is a \emph{sunspot $\ep$-equilibrium} in the game $\Gamma$
if
$\gamma^E_i(\xi) \geq \gamma^E_i(\xi'_i,\xi_{-i}) - \ep$,
for every player $i \in I$ and every strategy $\xi'_i$ of player~$i$.

Solan and Solan (2018) proved that every generalized quitting game admits a sunspot $\ep$-equilibrium, for every $\ep > 0$.
Our main result in this section is that when at least two players have at least two continue actions,
the game admits an $\ep$-equilibrium, for every $\ep > 0$.

\begin{theorem}
\label{theorem:game}
Let $\Gamma = (I,(A_i^c)_{i \in I},u)$ be a general quitting game
that satisfies $|A_1^c| \geq 2$ and $|A_2^c| \geq 2$.
Then for every $\ep > 0$ the game admits an $\ep$-equilibrium.
\end{theorem}

To prove Theorem~\ref{theorem:game} we describe the structure of the sunspot $\ep$-equilibrium
constructed in Solan and Solan (2018).
In that paper, it was proven that for every general quitting game $\Gamma = (I,(A_i^c)_{i \in I},u)$
there exists a mixed action profile $x =(x_i)_{i \in I}\in \times_{i \in I} \Delta(A_i)$
such that (at least) one of the following two alternatives hold for every $\ep > 0$:
\begin{enumerate}
\item[(A.1)]   The mixed action profile $x$ is absorbing, and,
when supplemented with threat strategies, it defines a stationary $\ep$-equilibrium.
\item[(A.2)]   The mixed action profile $x$ is nonabsorbing,
and the game admits a sunspot $\ep$-equilibrium $\xi$ in which at every stage~$t$ the players play the mixed action profile $x$,
except of possibly one player $i^t$,
whose identity is determined by the correlation device,
who plays the mixed action $(1-\eta^t)x_{i^t} + \eta^t Q_{i^t}$,
where the random variable $\eta^t$ has values in $(0,\ep)$ and depends on the history before stage $t$ and on $i^t$.
Moreover, under $\xi$ the play terminates with probability $1$.
\end{enumerate}

Thus, if Alternative~(A.2) holds,
then the players play mainly the stationary strategy profile $x$,
and take turns in stopping the game:
in each stage~$t$ the correlation device may designate one player $i^t$ as the possible quitter,
and that player stops the game with a history-dependent probability $\eta^t$.
If the correlation device did not designate any player as the possible quitter,
then all players follow $x$.
The order in which the players are selected by the correlation device is random, it depends on the device's past choices,
and is crafted so as to keep incentive constraints.

In both cases~(A.1) and~(A.2), statistical tests are conducted to ensure that the players do not deviate from the prescribed strategy profile.
In Case~(A.2) the players verify that the distribution of continue actions played by each player~$i$ is close to $x_i$.
In Case~(A.1), if under the strategy profile $x$ exactly one player, say, player~$i_0$, quits with positive probability,
then, if the play is not terminated after sufficiently many stages,
player~$i_0$ is punished.

\bigskip

\begin{proof}[Proof of Theorem~\ref{theorem:game}]
To prove the result we need to consider case (A.2) only.
Fix then $\ep > 0$ and let $\xi$ be a sunspot $\ep$-equilibrium in the extended game $\Gamma^E$
in which the players play mainly some nonabsorbing mixed action profile $x$.
Assume first that both $x_1$ and $x_2$ are not pure.

The idea is to define a strategy profile $\sigma$ in the game $\Gamma$
by replacing the correlation device with jointly controlled lotteries conducted by Players~1 and~2.
That is, we will divided the play into blocks of random size;
block~$t$ will correspond to stage~$t$ of the implementation of $\xi$.
All stages of the block except the last one will be used to perform a jointly controlled lottery by Players~1 and~2,
which will mimic the correlation device;
that is, in this lottery Players~1 and~2 will select a player $i^t \in I$ according to a probability distribution that is close to that indicated by $\xi$
for stage $t$.
In the last stage of the block the players will play as $\xi$ plays in stage~$t$,
given the outcome of the jointly controlled lottery conducted in that block.

Formally,
for each $t \in \dN$ denote by $k^t$ the stage of the game in which block~$t$ starts,
by $\widehat a^t$ the action profile that the players play in the last stage of block~$t$ (stage $k^{t+1}-1$),
and by $i^t$ the player who is selected by Players~1 and~2 in block~$t$ using the jointly controlled lottery mechanism of Theorem~\ref{theorem:strong}
(which will be described shortly in the context of the general quitting game).
Let $T \in \dN$ be sufficiently large such that
\begin{equation}
\label{equ:prob} \prob_\xi\left(\prod_{t =1}^T (1-\eta^t) > \ep\right) < \ep:
\end{equation}
under $\xi$ with probability at least $1-\ep$, the play terminates before stage~$T$ with high probability.

Let $\sigma$ be the following strategy profile in the general quitting game $\Gamma$:
\begin{itemize}
\item[(B.1)]
In block~$t$ the players play as follows.
Consider the situation of a jointly controlled lottery performed by Players~1 and~2,
where $C = \tfrac{\ep}{T^2}$, $J=I \cup\{0\}$, where $0$ will mean that no player is designated to quit,
and the distribution $\nu$ is the probability distribution over the set $I \cup\{0\}$
determined by the strategy profile $\xi$ given the past history $(\widehat a^1,\cdots,\widehat a^{t-1},i^1,\cdots,i^{t-1})$.

Under $\sigma$ the players play the mixed action profile $x$ until the game terminates
(if some player quits) or until stage $\tau_C$ of the block (stage $k^t+\tau_C-1$ of the game),
where $\tau_C$ is the stopping time defined in Eq.~\eqref{def:tau}.
Note that the length of this phase is uniformly bounded, even if one player deviates from the play described herein.
\item[(B.2)]
If the outcome of the jointly controlled lottery is 0,
in the last stage of the block the players play the mixed action profile $x$.
\item[(B.3)]
Otherwise, denote by $\widehat i^t \in I$ the player who is selected according to the mechanism described in Theorem~\ref{theorem:strong}.
At the last stage of the block,
the players follow the strategy $\xi$ at stage $t$,
given the history $(\widehat a^1,\cdots,\widehat a^{t-1},\widehat i^1,\cdots,\widehat i^t)$.
\end{itemize}

We thus defined a strategy profile $\sigma$ in the general quitting game $\Gamma$.
By Eq.~\eqref{equ:prob}
and since the difference between the distribution of the jointly controlled lottery at each block~$t$ and $\xi(h^t)$ is at most $\tfrac{\ep}{T}$,
a standard coupling argument shows that $\|\gamma(\sigma) - \gamma^E(\xi)\|_\infty \leq 2\ep$;
that is,
the expected payoff under $\sigma$ is $2\ep$-close to the expected payoff under $\xi$.

We argue that no player can profit more than $6\ep$ by deviating to a pure strategy.
Fix then a player $i \in I$ and a pure strategy $\sigma'_i$ of that player.
Using the strategy $\sigma'_i$ we will define a strategy $\xi'_i$ in the game with correlation device
and show that $\gamma_i(\sigma'_i,\sigma_{-i}) \leq \gamma^E_i(\xi'_i,\xi_{-i}) + 3\ep$.
Since $\xi$ is a sunspot $\ep$-equilibrium,
it will follow that
\[ \gamma_i(\sigma'_i,\sigma_{-i}) \leq \gamma^E_i(\xi'_i,\xi_{-i}) + 3\ep \leq \gamma^E_i(\xi) + 4\ep \leq \gamma_i(\sigma) + 6\ep, \]
as claimed.

%Given the strategy~$\sigma'_i$, let $\xi'_i$ be the following strategy of player~$i$.
Our goal now is to construct a strategy $\xi'_i$ in the game with correlation device
and prove that $\gamma_i(\sigma'_i,\sigma_{-i}) \leq \gamma^E_i(\xi'_i,\xi_{-i}) + 3\ep$.
As described above,
the strategy profile $(\sigma'_i,\sigma_{-i})$ defines a partition of the stages $\dN$ into blocks.%
\footnote{In fact, the partition is only of the stages up to the termination stage.}
For each block $t$ the play defines an element $\widehat i^t \in I \cup \{0\}$ that indicates if some player has to quit with low probability,
and if so, his identity,
and an action profile $\widehat a^t \in A$, which determines the action profile played by any player who does not quit.
Let $\rho_t$ be the conditional probability that under $(\sigma'_i,\sigma_{-i})$ player~$i$ quits during the first $\tau_C-1$ stages of block $t$,
given $\widehat i^1,\cdots,\widehat i^{t-1},\widehat a^1,\cdots,\widehat a^{t-1}$.
For every action $a_i \in A_i$, let $\mu_t(a_i)$ be the conditional probability that under $(\sigma'_i,\sigma_{-i})$ we have $\widehat a^t_i=a_i$,
given $\widehat i^1,\cdots,\widehat i^{t-1},\widehat i^t,\widehat a^1,\cdots,\widehat a^{t-1}$.
Let $\xi'_i$ be the strategy of player~$i$, that plays as follows at stage~$t$:
\begin{itemize}
\item   The quitting action $Q_i$ is played with probability $\rho_t$.
\item   For each $a_i \in A_i$, the action $a_i$ is played with probability $(1-\rho_t)\mu_t(a_i)$.
\end{itemize}
Since under $\xi_{-i}$ and $\sigma_{-i}$ the designated player quits with probability at most $\ep$,
it follows that
$\|\gamma(\sigma'_i,\sigma_{-i}) - \gamma^E(\xi'_i,\xi_{-i})\| \leq 3\ep$, as claimed.

\bigskip

It is left to take care of the situation that one (or both) of the mixed actions $x_1$ or $x_2$ is pure.
If the mixed action $x_1$ is pure, then, since $|A_1^c| \geq 2$,
we can find a mixed action $x'_1 \in \Delta(A_1^c)$ that is not pure and $\ep$-close to $x_1$ in the $l_\infty$-norm.
A similar statement holds for $x_2$.
In Step~(B.1) we then change $x_i$ by $x'_i$ for each player $i \in \{1,2\}$ whose mixed action $x_i$ is pure.
The only effect that this change has is that if a player quits, then his payoff changes by at most $2\ep$.
Consequently the strategy profile described above is a $10\ep$-equilibrium.
\end{proof}

\end{document}